\documentclass{amsart}
\usepackage{amssymb}
\usepackage{vaucanson-g}
\usepackage{pstricks} 
\usepackage{verbatim}

\newcommand{\ra}{\rightarrow}

\newcommand{\ul} {\underline}
\newcommand{\bl} {\begin{lemma}}
\newcommand{\el} {\end{lemma}}
\newcommand{\bt} {\begin{theorem}}
\newcommand{\et} {\end{theorem}} 

\newcommand{\Tkk} {\mathcal T _{k+1}}
\newcommand{\mc}{\mathcal}
\newcommand{\limn}{\lim_{n \ra \infty}}
\newcommand{\limk}{\lim_{k \ra \infty}}

\newcommand{\Tk}{\mathcal T _k}

\newcommand {\IR}{\mathbb R}
\newcommand {\IN}{\mathbb N}

\newcommand{\htop} {h_{top}}
\newcommand{\bp}{\begin{proof}}
\newcommand{\ep}{\end{proof}}
\newcommand {\be}{\begin{equation}}
\newcommand  {\ee} {\end{equation}}
\newcommand  {\beq} {\begin{eqnarray*}}
\newcommand  {\eeq} {\end{eqnarray*}}

\newcommand  {\bd} {\begin{definition}}
\newcommand  {\ed} {\end{definition}}

\newcommand{\dimh}{\mbox{Dim}_H}
\newcommand{\CU}{\mathcal U}

\newtheorem{theorem}{Theorem}[section]
\newtheorem{lemma}[theorem]{Lemma}
\newtheorem*{main}{Main Result 1}
\newtheorem*{main2}{Main Result 2}
\theoremstyle{definition}
\newtheorem{definition}[theorem]{Definition}

\theoremstyle{remark}
\newtheorem{remark}[theorem]{Remark}
\newtheorem*{Remarkstar}{Remark}
\numberwithin{equation}{section}

\begin{document}

\title[Irregular sets, the $\beta$-transformation and almost specification]{Irregular sets, the $\beta$-transformation and the almost specification property}


\author{Daniel J. Thompson}
\address{Department of Mathematics, The Pennsylvania State University, University Park, PA 16802, USA}
\email{thompson@math.psu.edu}





\begin{abstract}
Let $(X,d)$ be a compact metric space, $f:X \mapsto X$ be a continuous map satisfying a property we call almost specification (which is slightly weaker than the $g$-almost product property of Pfister and Sullivan), and $\varphi: X \mapsto \IR$ be a continuous function. We show that the set of points for which the Birkhoff average of $\varphi$ does not exist (which we call the irregular set) is either empty or has full topological entropy. Every $\beta$-shift satisfies almost specification and we show that the irregular set for any $\beta$-shift or $\beta$-transformation is either empty or has full topological entropy and Hausdorff dimension.
\end{abstract}

\maketitle
\section{Introduction}
For a compact metric space $(X, d)$, a continuous map $f:X \mapsto X$ and a continuous function $\varphi: X \mapsto \IR$, we define the irregular set for $\varphi$ to be
\[
\widehat X(\varphi, f) := \left \{ x \in X : \lim_{n \ra \infty} \frac{1}{n} \sum_{i = 0}^{n-1} \varphi (f^i (x)) \mbox{ does not exist } \right \}.
\]
The irregular set arises naturally in the context of multifractal analysis. As a consequence of Birkhoff's ergodic theorem, the irregular set is not detectable from the point of view of an invariant measure. However, it is an increasingly well known phenomenon that the irregular set can be large from the point of view of dimension theory \cite{Bad}. Symbolic dynamics methods have confirmed this in the uniformly hyperbolic setting \cite{BS}, for certain non-uniformly hyperbolic examples \cite{PW} and for a large class of multimodal maps \cite{To}. The irregular set has also been the focus of a great deal of work by Olsen and collaborators, for example in \cite{Ol2} and \cite{OW}. The papers \cite{Tho4, EKL} studied the irregular set for maps with the specification property using a non-symbolic method inspired by the work of Takens and Verbitskiy \cite{TV}. 

Ruelle uses the terminology `set of points with historic behaviour' to describe the irregular set \cite{Ruelle}. The idea is that points for which the Birkhoff average does not exist are capturing the `history' of the system, whereas points whose Birkhoff average converge only see average behaviour. For example, in the dynamics of the weather, the irregular points are the ones that have observed epochs of climate change. In \cite{Takens}, in contrast to the dimension-theoretic point of view, Takens asks for which smooth dynamical systems the irregular set has positive Lebesgue measure. 

Our current work, which takes the dimension-theoretic viewpoint, is inspired by a recent innovation of Pfister and Sullivan \cite{PfS}, who have introduced a weak specification property called the $g$-almost product property, which we take the liberty of renaming as the almost specification property. 
They extended the variational principle of Takens and Verbitskiy \cite{TV} to the class of maps which satisfy this property. This variational principle is a key tool in multifractal analysis and has been extended by the author in a different direction in \cite{Tho6}. 

A striking application of the almost specification property is that it applies to every $\beta$-shift. In sharp contrast, the set of $\beta$ for which the $\beta$-shift has the specification property has zero Lebesgue measure \cite{Bu}, \cite{Schm}. The $\beta$-shift is the natural symbolic space associated to the $\beta$-transformation $f_\beta (x) = \beta x \mod 1$, $\beta \in (1, \infty), x \in [0,1)$. The $\beta$-transformation has been widely studied since its introduction in 1957 by Renyi \cite{Ren}. The sustained interest in the study of the $\beta$-tranformation arises from its connection with number theory and its special role as a model example of a one-dimensional expanding dynamical system which admits discontinuities. 

\begin{main}
When $f$ satisfies the almost specification property, the irregular set is either empty or has full topological entropy. 
\end{main}
\begin{main2}
The irregular set for an arbitrary $\beta$-transformation (or $\beta$-shift) is either empty or has full entropy $\log \beta$ and Hausdorff dimension $1$.
\end{main2}
Our second main result (stated formally as theorem \ref{theorem.a.0.2} and theorem \ref{theorem.ae.beta}) is a corollary of our first main result (stated formally as theorem \ref{theorem.a.0.1}). We work through the application to the $\beta$-shift carefully.  We emphasise that our first main result applies to many systems other than $\beta$-transformations. For example, every ergodic toral automorphism has almost specification. This is a special case of the result that every automorphism of a compact metric abelian group that is ergodic (with respect to Haar measure) and has finite topological entropy has almost specification \cite{Mar, Dat, Ya}. This provides another important class of examples where our first main theorem yields new results.

To undertake the proof of our first main result, we develop notions of almost spanning sets, strongly separated sets and a generalised version of the Katok formula for entropy. These should be of independent interest.  We also include an alternative proof of our second main result based on approximating the  $\beta$-shift from inside by shifts of finite type. We discuss the merits of the two different approaches in \S \ref{sfc}.

In \cite{Tho4}, we showed that when $f$ has the specification property, the irregular set is either empty or has full topological pressure. The method of this paper can be used to show that this more general result holds true in the almost specification setting. We restrict ourselves to the special case of entropy for clarity and brevity.

In \S \ref{beta.2}, we give definitions. In \S \ref{beta.1.2} we establish our general version of the Katok formula for entropy. In \S \ref{beta.3}, we prove our first main result. In \S \ref{beta.4}, we consider arbitrary  $\beta$-shifts and $\beta$-transformations and establish our second main result.
\begin{Remarkstar}
I am grateful to a referee of this paper who informed me of the work of Durner, which I was not aware of when this work was completed. Our result on $\beta$-shifts,  although not stated explicitly in Durner's work, can be obtained as a corollary of \cite[Theorem 4]{Dur}. Durner's work does not apply in a general setting but contains some very interesting theorems concerning $\beta$-shifts. Durner's proof implicitly uses the almost specification property of the $\beta$-shift and may be the first work in which this property was exploited. 
\end{Remarkstar}

\section{Definitions and preliminaries} \label{beta.2}
Let $(X,d)$ be a compact metric space and $f:X \mapsto X$ a continuous map 
Let $C(X)$ denote the space of continuous functions from $X$ to $\IR$, and $\varphi \in C(X)$. Let $S_n \varphi (x) := \sum_{i = 0}^{n-1} \varphi (f^i (x))$. Let $\mathcal{M}_{f} (X)$ denote the space of $f$-invariant probability measures and $\mathcal{M}^e_{f} (X)$ denote those which are ergodic.
\subsection{The irregular set}
We define the irregular set for $\varphi$ to be
\be \label{ir}
\widehat X(\varphi, f) := \left \{ x \in X : \lim_{n \ra \infty} \frac{1}{n} \sum_{i = 0}^{n-1} \varphi (f^i (x)) \mbox{ does not exist } \right \}.
\ee
For a constant $c$, let $\mbox{Cob} (X, f, c)$ denote the space of functions cohomologous to $c$ and $\overline {\mbox{Cob} (X, f, c)}$ be the closure of $\mbox{Cob} (X, f, c)$ in the sup norm. In \cite{Tho4}, we proved the following lemma. 
\bl \label{equiv}
The following are equivalent for $\varphi \in C(X)$:


(a) $\frac{1}{n}S_n \varphi$ does not converge pointwise to a constant;

(b) $\inf_{\mu \in \mathcal{M}_{f} (X)} \int \varphi d \mu < \sup_{\mu \in \mathcal{M}_{f} (X)} \int \varphi d \mu$;

(c) $\inf_{\mu \in \mathcal{M}^e_{f} (X)} \int \varphi d \mu < \sup_{\mu \in \mathcal{M}^e_{f} (X)} \int \varphi d \mu$;

(d) $\varphi \notin \bigcup_{c \in \IR} \overline {\mbox{Cob} (X, f, c)}$;

(e) $\frac{1}{n}S_n \varphi$ does not converge uniformly to a constant.
\el
We also proved that if $\widehat X(\varphi, f)$ is non-empty, then the properties of lemma \ref{equiv} hold. It is a corollary of our main result (theorem \ref{theorem.a.0.1}) that when $f$ has the almost specification property, (b) implies that $\widehat X(\varphi, f)$ is non-empty.
\subsection{The almost specification property}
Pfister and Sullivan have introduced a property called the $g$-almost product property. We take the liberty of renaming this property as the almost specification property. The almost specification property can be verified for every $\beta$-shift (see \S \ref{beta.4.1}). The version we use here is a priori weaker than that in \cite{PfS} and is slightly weaker than the almost product property of \cite{Ya}. First we need an auxiliary definition.
\begin{definition}
Let $\epsilon_0 >0$. A function $g: \IN \times (0, \epsilon_0) \mapsto \IN$ is called a mistake function if for all $\epsilon \in(0, \epsilon_0)$ and all $n \in \IN$, $g(n, \epsilon) \leq g(n+1, \epsilon)$ and 
\[
\limn \frac{g(n, \epsilon)}{n} = 0.
\]
Given a mistake function $g$, if $\epsilon > \epsilon_0$, we define $g(n, \epsilon) := g(n, \epsilon_0)$. 
\end{definition}
We note that for fixed $k \in \IN$ and $\lambda >0$, if $g$ is a mistake function, then so is $h$ defined by $h(n, \epsilon)=kg(n, \lambda \epsilon)$. 
\begin{definition}
For $n,m \in \IN, m<n$, we define the set of $(n, -m)$ index sets to be
\[
I(n,-m) := \{ \Lambda \subseteq \{0, \ldots, n-1 \}, \#\Lambda \geq n - m\}.
\]
Let $g$ be a mistake function and $\epsilon > 0$. For $n$ sufficiently large so that $g(n, \epsilon) <n$, we define the set of $(g, n, \epsilon)$ index sets to be $I(g;n, \epsilon) := I(n, -g(n, \epsilon))$. Equivalently, 
\[
I(g;n, \epsilon) := \{ \Lambda \subseteq \{0, \ldots, n-1 \}, \# \Lambda \geq n - g(n, \epsilon)\}.
\] 
For a finite set of indices $\Lambda$, we define
\[
d_\Lambda(x, y) = \max \{d(f^j x, f^j y) : j \in \Lambda \} \mbox{ and } B_\Lambda (x, \epsilon) = \{ y \in X : d_\Lambda(x, y) < \epsilon\}.
\]
\end{definition}
\begin{definition}
When $g(n, \epsilon) < n$, we define a `dynamical ball of radius $\epsilon$ and length $n$ with $g(n, \epsilon)$ mistakes'. Let 
\beq
B_n (g; x, \epsilon) &:=& \{ y \in X : y \in B_\Lambda (x, \epsilon) \mbox{ for some } \Lambda \in I(g;n, \epsilon)\}\\
&=& \bigcup_{\Lambda \in I(g;n, \epsilon)} B_\Lambda (x, \epsilon)
\eeq
\end{definition}
\begin{definition} 
A continuous map $f: X \mapsto X$ satisfies the almost specification property if there exists a mistake function $g$ such that for any $\epsilon_1, \ldots ,\epsilon_k >0$, there exist integers $N(g,\epsilon_1), \ldots, N(g,\epsilon_k)$ such that for any $x_1, \ldots, x_k$ in $X$ and integers $n_i \geq N(g, \epsilon_i)$,
\[
\bigcap_{j=1}^k f^{-\sum_{i=0}^{j-1} n_i} B_{n_j} (g; x_j, \epsilon_j) \neq \emptyset,
\]
where $n_0 =0$.
\end{definition}
\begin{remark}
The function $g$ can be interpreted as follows. The integer $g(n, \epsilon)$ tells us how many mistakes we are allowed to make when we use the almost specification property to $\epsilon$ shadow an orbit of length $n$. Henceforth, we assume for convenience and without loss of generality that $N(g, \epsilon)$ is chosen so that $g(n, \epsilon)/n < 0.1$ for all $n\geq N(g, \epsilon)$.
\end{remark}
\begin{remark}
Pfister and Sullivan use a slightly different definition of mistake function (which they call a blowup function). They do not allow $g$ to depend on $\epsilon$.  An example of a function which is a mistake function under our definition but is not considered by Pfister and Sullivan is $g(n, \epsilon) = \epsilon^{-1} \log n$.  
Since we allow a larger class of mistake functions, the almost specification property defined here is slightly more general than the $g$-almost product property of Pfister and Sullivan. 
\end{remark}
\begin{definition} \label{3a}
A continuous map $f: X \mapsto X$ satisfies the specification property if for all $\epsilon > 0$, there exists an integer $m = m(\epsilon )$ such that for any collection $\left \{ I_j = [a_j, b_j ] \subset \IN : j = 1, \ldots, k \right \}$ of finite intervals with $a_{j+1} - b_j \geq m(\epsilon ) \mbox{ for } j = 1, \ldots, k-1 $ and any $x_1, \ldots, x_k$ in $X$, there exists a point $x \in X$ such that
\begin{equation} \label{3a.02}
d(f^{p + a_j}x, f^p x_j) < \epsilon \mbox{ for all } p = 0, \ldots, b_j - a_j \mbox{ and every } j = 1, \ldots, k.
\end{equation}
\end{definition}
Pfister and Sullivan showed that the specification property implies the almost specification property \cite{PfS} using ANY blow-up function $g$. To see the relation between the two concepts, we note that if $f$ has specification and we set $g(n, \epsilon) = m(\epsilon)$ for all $n$ larger than $m(\epsilon)$ and set $N(g, \epsilon) = m(\epsilon) +1$, then for any $x_1, \ldots, x_k$ in $X$ and integers $n_i \geq N(g, \epsilon)$, we have
\[
\bigcap_{j=1}^k f^{-\sum_{i=0}^{j-1} n_i} B_{n_j} (g; x_j, \epsilon) \neq \emptyset.
\]
The trick required to replace $\epsilon$ by $\epsilon_1, \ldots, \epsilon_k$ can be found in \cite{PfS}.

\begin{remark}
In \cite{Ya}, Yamamoto studies the relationship between a number of weak forms of the specification property, including the Pfister and Sullivan definition.  We note that our definition of almost specification is different from all the properties defined in \cite{Ya}, and the almost specification property of \cite[Definition 2.1]{Ya} is a special case of ours.

\end{remark}
\section{A modified Katok entropy formula} \label{beta.1.2}
The following definitions of `strongly separated' and `almost spanning' sets are inspired by Pfister and Sullivan and designed for use in the setting of maps with the almost specification property.
\begin{definition}
Let $Z \subseteq X$. For $m <n$, a set $S$ is $(n, -m, \epsilon)$ separated for $Z$ if $S \subset Z$ and for every $\Lambda \in I(n, -m)$, we have $d_\Lambda(x,y) > \epsilon$ for every $x,y \in S$. We define a set $S$ to be $(g; n, \epsilon)$ separated if it is $(n, -g(n, \epsilon), \epsilon)$ separated. Equivalently, $S$ is $(g; n, \epsilon)$ separated if for every $x,y \in S$
\[
\# \{j \in \{0, \ldots n-1\} : d(f^jx, f^jy) > \epsilon \} >g(n, \epsilon).
\]
\end{definition}
We think of an $(n, -m, \epsilon)$ separated set to be `a set which remains $(n, \epsilon)$ separated when you permit $m$ mistakes'. In particular, a set $S$ which is $(g; n, \epsilon)$ separated is $(n, \epsilon)$ separated in the usual sense. We define the natural dual notion of a $(g; n, \epsilon)$ spanning set.
\begin{definition}
For $m <n$, a set $S \subset Z$ is $(n, -m, \epsilon)$ spanning for $Z$ if for all $x \in Z$, there exists $y \in S$ and $\Lambda \in I(n, -m)$ such that $d_\Lambda (x, y) \leq \epsilon$. Note that $\Lambda$ depends on $x$ and an $(n, \epsilon)$ spanning set is always $(n, -m, \epsilon)$ spanning.
We define a set $S$ to be $(g; n, \epsilon)$ spanning if it is $(n, -g(n, \epsilon), \epsilon)$ spanning. 
\end{definition}

We think of an $(n, -m, \epsilon)$ spanning set to be `a set which requires up to $m$ mistakes to be $(n, \epsilon)$ spanning'. Let
\[
s_n(g;Z, \epsilon) = \sup \{ \#S : S \mbox{ is $(g;n,\epsilon)$ separated for $Z$}\},
\]
\[
r_n(g;Z, \epsilon) = \inf \{ \#S : S \mbox{ is $(g;n,\epsilon)$ spanning for $Z$}\},
\]
\[
s_n(Z, \epsilon) = \sup \{ \#S : S \mbox{  is $(n,\epsilon)$ separated for $Z$}\},
\]
\[
r_n(Z, \epsilon) = \inf \{ \#S : S \mbox{ is $(n,\epsilon)$ spanning for $Z$}\}.
\]
\begin{lemma}
Let $g$ be any mistake function and let $h(n, \epsilon) = 2 g(n, \epsilon/2)$. We have

(1) $r_n(g;Z, \epsilon) \leq s_n(g;Z, \epsilon) \leq s_n(Z, \epsilon)$,

(2) $s_n(h;Z, 2\epsilon) \leq r_n(g; Z, \epsilon) \leq r_n (Z, \epsilon) \leq s_n (Z, \epsilon)$. 
\end{lemma}
\bp
Suppose that $S$ is a $(g;n,\epsilon)$ separated set for $Z$ of maximum cardinality such that $S$ is not $(g;n,\epsilon)$ spanning. 
We can find
\[
z \in Z \setminus \bigcup_{\Lambda \in I(g;n, \epsilon)} \bigcup_{x \in S} \overline{B_\Lambda(x, \epsilon)} = \bigcap_{\Lambda \in I(g;n, \epsilon)} \left( Z \setminus \bigcup_{x \in S} \overline{B_\Lambda(x, \epsilon)} \right ).
\]
Since $d_\Lambda (x,z) > \epsilon$ for all $x \in S$ and $\Lambda \in I(g;n, \epsilon)$, then $S \cup \{z\}$ is a $(g;n,\epsilon)$ separated set, which contradicts the maximality of $S$. Thus, every $(g;n,\epsilon)$ separated set of maximal cardinality is $(g;n,\epsilon)$ spanning.

For (2), suppose $E$ is $(h; n, 2\epsilon)$ separated and $F$ is $(g; n, \epsilon)$ spanning for $Z$. Define $\phi:E \mapsto F$ by choosing for each $x \in E$ some $\phi(x) \in F$ and some $\Lambda_x \in I(g;n, \epsilon)$ such that $d_{\Lambda_{x}} (x, \phi(x)) \leq \epsilon$. Suppose $x \neq y$. Let $\Lambda = \Lambda_{x} \cap \Lambda_{y}$. Since $\# \Lambda \geq n- 2g(n, \epsilon)$, then $\Lambda \in I(h;n, 2 \epsilon)$. We have $d_\Lambda(\phi(x), \phi(y)) > 0$ and thus $\phi(x) \neq \phi(y)$ . Thus $\phi$ is injective and hence $|E| \leq |F|$. 
\ep

\begin{theorem} [Modified Katok entropy formula] \label {theorem4}
Let $(X,d)$ be a compact metric space, $f:X \mapsto X$ be a continuous map and $\mu$ be an ergodic invariant measure. For $\gamma \in (0, 1)$ and any mistake function $g$, we have  
\[
h_\mu = \lim_{\epsilon \ra 0} \liminf_{n \ra \infty} \frac{1}{n} \log ( \inf \{ r_n(g; Z, \epsilon): Z \subset X, \mu(Z) \geq 1 - \gamma \}) 
\]
The formula remains true if we replace the $\liminf$ by $\limsup$ and/or $r_n(g; Z, \epsilon)$ by $s_n(g; Z, \epsilon)$. The value taken by the $\liminf$ can be bounded below independent of the choice of mistake function $g$.
\end{theorem}

\bp
Since $r_n(Z, \epsilon) \geq r_n(g; Z, \epsilon)$, it follows from the original Katok entropy formula that the expression on the right hand side is less than or equal to $h_\mu$
(this is the easier inequality to prove directly any way). To prove the opposite inequality, we give a method inspired by the proof of theorem A2.1 of \cite{Pe}.

For any $\eta > 0$, there exists $\delta$, $0 < \delta \leq \eta$, a finite Borel partition $\xi = \{ C_1, \ldots, C_m\}$ and a finite open cover $\CU = \{U_1, \ldots, U_k\}$ of $X$ where $k \geq m$ with the following properties:

(1) $\mbox{Diam} (U_i) \leq \eta$, $\mbox {Diam}(C_j) \leq \eta$ for all $ i = 1, \ldots, m$, $j = 1, \ldots, k$,

(2) $ \overline U_i \subset C_i$ for all $i = 1, \ldots, m$,

(3) $\mu (C_i \setminus U_i) \leq \delta$ for all $i = 1, \ldots, m$ and $\mu ( \bigcup_{i = m +1}^{k} U_i) \leq \delta$,

(4) $2\delta \log m \leq \eta$.

This is a consequence of the regularity of the measure $\mu$. 
Fix $\eta$ so $ 1-\gamma > \eta > 0$ and take the corresponding number $\delta$, covering $\CU$ and partition $\xi$. Fix $Z \subset X$ with $\mu(Z) > 1-\gamma$. 
Let $t_n (x)$ denote the number of $l$, $0 \leq l \leq n - 1$ for which $f^l (x) \in \bigcup_{i = m + 1}^{k}U_i$. Let $\xi_n = \bigvee_{i = 0} ^{n-1} f^{-i} \xi$ and $C_{\xi_n} (x)$ denote the member of the partition $\xi_n$ to which $x$ belongs.
\begin{lemma}
There exists a set $A \subset Z$ and $N >0$ with $\mu (A) \geq \mu (Z) - \delta$ such that for every $x \in A$ and $n \geq N$

(1) $t_n (x) \leq 2 \delta n$

(2) $\mu (C_{\xi_n}(x)) \leq \exp(-(h_\mu (f, \xi) - \delta) n)$

\end{lemma}
\begin{proof}
Let $\chi$ be the characteristic function of $\bigcup_{i = m + 1}^{k}U_i$. We can write $t_n (x) = \sum_{i=0}^{n-1} \chi(f^i x)$. 
By Birkhoff's ergodic theorem and Egorov's theorem, we can find a set $A_1 \subset X$ with $\mu(A_1) \geq \mu(Z) - \frac{\delta}{2}$ such that for $x \in A_1$, we have uniform convergence 
\[
n^{-1} t_n(x) = \frac{1}{n} \sum_{i=0}^{n-1} \chi (f^{i}x) \rightarrow \int \chi d\mu = \mu(\bigcup_{i = m + 1}^{k}U_i) \leq \delta.
\]
Choose $N_1$ such that if $n \geq N_1$ and $x \in A_1$, then $t_n (x) \leq 2 \delta n$.

By Shannon-McMillan-Breiman theorem and Egorov's theorem, we can find a set $A_2 \subset X$ with $\mu(A_2) \geq \mu(Z) - \frac{\delta}{2}$ such that for $x \in A_1$, we have uniform convergence 
\[
-\frac{1}{n} \log \mu (C_{\xi_n} (x)) \rightarrow h_\mu(f, \xi)
\]
There exists $N_2$ such that if $n \geq N_2$ and $x \in A_2$, then $-\frac{1}{n} \log \mu (C_{\xi_n} (x)) \leq h(f, \xi) + \delta$. Set $A = A_1 \cap A_2$ and $N = \max\{N_1, N_2\}$ and the lemma is proved.  
\end{proof}
Let $\xi_n^{\ast}$ be the collection of elements $C_{\xi_n}$ of the partition $\xi_n$ for which $C_{\xi_n} \cap A \neq \emptyset$. 
Then for $n \geq N$, using property (2) of $A$,
\beq
\# \xi_n^{\ast} &\geq& \sum_{C \in \xi_n^{\ast}} \mu (C) \exp \{n (h_\mu (f, \xi) - \delta)\}\\
&\geq& \mu (A) \exp \{n (h_\mu (f, \xi) - \delta)\}\}.
\eeq
Let $2\epsilon$ be a Lebesgue number for $\CU$ and let $S$ be $(g; n, \epsilon)$ spanning for $Z$. 
We have $Z \subseteq \bigcup_{x \in S} \bigcup_{\Lambda \in I(g;n, \epsilon)} \overline{B}_{\Lambda} (x, \epsilon)$.
Let us fix $B = \overline{B}_{\Lambda} (x, \epsilon)$ for some $x \in S$ and $\Lambda \in I(g;n, \epsilon)$. Let $\xi_\Lambda$ be the partition $\bigvee_{i \in \Lambda} f^{-i} \xi$. We estimate the number $p(B,\xi_{\Lambda})$ of elements of the partition $\xi_{\Lambda}$ which have non-empty intersection with $A \cap B$.

Since $2 \epsilon$ is a Lebesgue number for $\CU$, then $\overline{B} (f^j x, \epsilon) \subset U_{i_j}$ for some $U_{i_j} \in \CU$. If $i_j \in \{1, \ldots, m\}$ then $f^{-j} (U_{i_j}) \subset f^{-j} (C_{i_j})$. If $i_j \in \{m + 1, \ldots, k\}$, then anything up to $m$ sets of the form $f^{-j} (C_{i_j})$ may have non-empty intersection with $f^{-j} (U_{i_j})$. 
It follows, using property (1) of $A$, that
\[
p(B,\xi_{\Lambda}) \leq m^{2\delta n } = \exp(2\delta n \log m)
\]
The number $p(B,\xi_{n})$ of elements of the partition $\xi_{n}$ which have non-empty intersection with both $A$ and $B$ satisfies
\[
p(B,\xi_{n}) \leq p(B,\xi_{\Lambda}) m^{g(n, \epsilon)} \leq \exp \{(2\delta n + g(n, \epsilon)) \log m\}.
\]
It follows that\[
\#\xi_n^{\ast} \leq \sum_{x \in S} \sum_{\Lambda \in I(g;n, \epsilon)} p(\overline{B}_{\Lambda} (x, \epsilon),\xi_n) \leq  G_{n, \epsilon}\#S  \exp \{(2\delta n + g(n, \epsilon)) \log m\},
\]
where  $G_{n, \epsilon}  = \sum_{i \leq g(n, \epsilon)} {n\choose i}$. Rearranging, we have
\[
\frac{1}{n}\log \#S \geq h_\mu (f, \xi) - \delta - \left(2\delta + \frac{g(n, \epsilon)}{n} \right) \log m - \frac{1}{n} \log G_{n, \epsilon}  .
\]
A short combinatorial argument (see \cite[Lemma 2.1]{PfS2}) shows that $\frac{1}{n} \log G_{n, \epsilon}  \ra 0$. Since $2\delta \log m < \eta$, $\mbox{Diam}(\xi) < \eta$, $\delta < \eta$, $\frac{g(n, \epsilon)}{n} \ra 0$, and $\eta$ was arbitrary, we are done.
\ep
As a corollary, we have a version of theorem \ref{theorem4} for topological entropy (which we do not use).
\begin{theorem}
Let $(X,d)$ be a compact metric space and $f:X \mapsto X$ be a continuous map. We have  
\[
\htop (f) = \lim_{\epsilon \ra 0} \liminf_{n \ra \infty} \frac{1}{n} \log  r_n(g; X, \epsilon) 
\]
The formula remains true if we replace the $\liminf$ by $\limsup$ and/or $r_n(g; X, \epsilon)$ by $s_n(g; X, \epsilon)$. The value taken by the $\liminf$ (or $\limsup$) is independent of the choice of mistake function $g$.
\end{theorem}

\subsection{Topological entropy for non-compact sets}
We consider the Bowen definition of topological entropy for non-compact sets. Following Pesin and Pitskel \cite{PP2}, we give a definition which is suitable for maps which have discontinuities. Suppose $X$ is a compact metric space, $Y$ is a (generally non-compact) subset of $X$ and $f:Y \mapsto Y$ is continuous. When $f:X \mapsto X$ is continuous (which is the setting of our main theorem), we set $Y=X$. In our final section, when we consider the $\beta$-transformation $f_\beta$, we set 
\[
Y= X \setminus \{ \beta^{-i}: i \in \IN\} = X \setminus \bigcup_i f_\beta^{-i}(0).
\]
Let $Z \subset Y$ be an arbitrary Borel set, not necessarily compact or invariant. 
We consider finite and countable collections of the form $\Gamma = \{ B_{n_i}(x_i, \epsilon) \}_i$. 
For $\alpha \in \IR$, we define the following quantities:
\[
Q(Z,\alpha, \Gamma) = \sum_{B_{n_i}(x_i, \epsilon) \in \Gamma} \exp -\alpha n_i,
\]
\[
M(Z, \alpha, \epsilon, N) = \inf_{\Gamma} Q(Z,\alpha, \Gamma),
\]
where the infimum is taken over all finite or countable collections of the form $\Gamma = \{ B_{n_i}(x_i, \epsilon) \}_i$ with $x_i \in X$ such that $\Gamma$ covers Z and $n_i \geq N$ for all $i = 1, 2, \ldots$. Define
\[
m(Z, \alpha, \epsilon) = \lim_{N \rightarrow \infty} M(Z, \alpha,\epsilon, N).
\]
The existence of the limit is guaranteed since the function $M(Z, \alpha,\epsilon, N)$ does not decrease with N. By standard techniques, we can show the existence of
\begin{displaymath}
\htop (Z, \epsilon) := \inf \{ \alpha : m(Z, \alpha, \epsilon) = 0\} = \sup \{ \alpha :m(Z, \alpha, \epsilon) = \infty \}.
\end{displaymath}
\begin{definition}
The topological entropy of $Z$ is given by
\[
\htop (Z) = \lim_{\epsilon \ra 0} \htop (Z, \epsilon).
\]
\end{definition}
See \cite{Pe} for verification that the quantities $\htop (Z, \epsilon)$ and $\htop (Z)$ are well defined. When $X=Y$, we denote the topological entropy of the dynamical system $(X, f)$ by $\htop (f)$ and we note that $\htop(X)=\htop (f)$. We sometimes write $\htop(Z, f)$ in place of $\htop(Z)$ when we wish to emphasise the dependence on $f$.
\subsection{Topological entropy for shift spaces} \label{tes}
Let $\Sigma$ be a one-sided shift space (i.e. a closed, $\sigma$-invariant subset of a full one-sided shift on finitely many symbols). For shift spaces, the definition of topological entropy can be simplified and we introduce notation that reflects this. For $x=(x_i)_{i=1}^\infty$, let $C_n (x) = \{ y \in \Sigma : x_i = y_i \mbox{ for } i=1, \ldots, n\}$.  Let $Z \subset \Sigma$ be an arbitrary Borel set, not necessarily compact or invariant. We consider finite and countable collections of the form $\Gamma = \{ C_{n_i}(x_i) \}_i$. 
For $s \in \IR$, we define the following quantities:
\[
Q(Z,s, \Gamma) = \sum_{C_{n_i}(x_i) \in \Gamma} \exp -s n_i,
\]
\[
M(Z, s, N) = \inf_{\Gamma} Q(Z,s, \Gamma),
\]
where the infimum is taken over all finite or countable collections of the form $\Gamma = \{ C_{n_i}(x_i) \}_i$ with $x_i \in \Sigma$ such that $\Gamma$ covers $Z$ and $n_i \geq N$ for all $i = 1, 2, \ldots$. Define
\[
m(Z, s) = \lim_{N \rightarrow \infty} M(Z, s, N).
\]
The existence of the limit is guaranteed since the function $M(Z, s, N)$ does not decrease with $N$. 
\bl
The topological entropy of $Z \subset \Sigma$ is given by
\[
\htop (Z) := \inf \{ s : m(Z, s) = 0\} = \sup \{ s :m(Z, s) = \infty \}.
\]
\el
The proof, which we omit, follows from the fact that every open ball $B(x, \epsilon)$ in $\Sigma$ is a set of the form $C_n (x)$.

\section{Main result} \label{beta.3}
\begin{theorem} \label {theorem.a.0.1}
Let $(X,d)$ be a compact metric space and $f:X \mapsto X$ be a continuous map with the almost specification property. Assume that $\varphi \in C(X)$ satisfies $\inf_{\mu \in \mathcal{M}_{f} (X)} \int \varphi d \mu < \sup_{\mu \in \mathcal{M}_{f} (X)} \int \varphi d \mu$. Let $\widehat X(\varphi, f)$ be the irregular set for $\varphi$ defined as in (\ref{ir}), then $h_{top} (\widehat X(\varphi, f)) =  h_{top}(f)$.
\end{theorem}
We remark that $\widehat X(\varphi, f) \neq \emptyset$ is a sufficient condition on $\varphi$ for the theorem to apply.
\bp 
Let us fix a small $\gamma > 0$, and take ergodic measures $\mu_1$ and $\mu_2$ such that

(1) $h_{\mu_i} > \htop (f) - \gamma$ for $i = 1, 2$, 

(2) $\int \varphi d \mu_1 \neq \int \varphi d\mu_2$. 

That we are able to choose $\mu_i$ to be ergodic is a slightly subtle point. Let $\mu_1$ be ergodic and satisfy $h_{\mu_1} > \htop(f) - \gamma /3$. Let $\nu \in \mc M_f (X)$ satisfy $\int \varphi d {\mu_1} \neq \int \varphi d \nu$. Let $\nu^\prime = t \mu_1 + (1-t) \nu$ where $t \in (0,1)$ is chosen sufficiently close to $1$ so that $h_{\nu^\prime} > \htop(f) - 2 \gamma /3$. By \cite{PfS2}, when $f$ has the almost specification property, we can find a sequence of ergodic measures $\nu_n \in \mc M_f (X)$ such that $h_{\nu_n} \ra h_{\nu^\prime}$ and $\nu_n \ra \nu^\prime$ in the weak-$\ast$ topology. Therefore, we can choose a measure belonging to this sequence which we call $\mu_2$ which satisfies $h_{\mu_2} > \htop(f) - \gamma$ and $\int \varphi d \mu_1 \neq \int \varphi d\mu_2.$ (We could avoid the use of the result from \cite{PfS2} by giving a self-contained proof along the lines of the `modified construction' in \cite{Tho4}. We do not do so in the interest of brevity.)
Choose $\delta > 0$ sufficiently small so
\[
\left | \int \varphi d \mu_1 - \int \varphi d\mu_2 \right | > 4 \delta.
\]
Let $\rho : \IN \mapsto \{1, 2\}$ be given by $\rho(k) = (k+1) (\mbox{mod} 2) +1$. Choose a strictly decreasing sequence $\delta_k \ra 0$ with $\delta_1 < \delta$ and a strictly increasing sequence $l_k \ra \infty$ so the set
\be \label{5}
Y_k := \left \{ x \in X^\prime : \left | \frac{1}{n} S_n \varphi (x) - \int \varphi d\mu_{\rho(k)} \right | < \delta_k \mbox{ for all } n \geq l_k\right \}
\ee
satisfies $\mu_{\rho(k)} (Y_k) > 1 - \gamma$ for every $k$. 

The following lemma follows readily from Theorem \ref{theorem4}. 
\begin{lemma} \label {lemma1}
Define mistake functions $h_k (n, \epsilon) := 2g(n, \epsilon /2^k)$. For any sufficiently small $\epsilon > 0$, we can find a sequence $n_k \ra \infty$ and a countable collection of finite sets $\mc S_k$ so that each $\mc S_k$ is a $(h_k; n_k, 4 \epsilon)$ separated set for $Y_k$ and satisfies
\[
\# S_k \geq \exp( n_k (\htop (f) - 4 \gamma)).
\]
Furthermore, the sequence $n_k$ can be chosen so that $n_k \geq l_k$, 
$n_k > N(h_k, \epsilon)$ and $h_k(n_k, \epsilon)/n_k \ra 0$.
\end{lemma}
We choose $\epsilon$ sufficiently small 
and fix all the ingredients provided by lemma \ref{lemma1}. Our strategy is to construct a certain fractal $F \subset \widehat X(\varphi, f)$, on which we can define a sequence of measures suitable for an application of the following result of Takens and Verbitskiy \cite{TV}. 
\bt [Entropy Distribution Principle] 
Let $f : X \mapsto X$ be a continuous transformation. Let $Z \subseteq X$ be an arbitrary Borel set. Suppose there exists $\epsilon > 0$ and $s \geq 0$ such that one can find a sequence of Borel probability measures $\mu_{k}$, a constant $K > 0$ and an integer $N$ satisfying 
\[
\limsup_{k \ra \infty} \mu_{k} (B_n (x, \epsilon)) \leq K e^{-ns}
\] 
for every ball $B_n (x, \epsilon)$ such that $B_n (x, \epsilon) \cap Z \neq \emptyset$ and $n \geq N$. Furthermore, assume that at least one limit measure $\nu$ of the sequence $\mu_{k}$ satisfies $\nu (Z) > 0$. Then $\htop (Z, \epsilon) \geq s$.
\et

\subsection{Construction of the fractal F}
Let us choose a sequence with $N_0 = 0$ and $N_k$ increasing to $\infty$ sufficiently quickly so that
\begin{equation} \label{f.1}
\limk \frac{n_{k+1}}{N_k} = 0,  \limk \frac{N_1 n_1 + \ldots + N_k n_k}{N_{k+1}} = 0.
\end{equation}
Let $\ul x_i =(x_1^i, \ldots, x_{N_i}^i) \in S_i^{N_i}$. For any $(\ul x_1, \ldots, \ul x_k) \in S_1^{N_1} \times \ldots \times S_k^{N_k}$, by the almost specification property, we have
\[
B(\ul x_1, \ldots \ul x_k) := \bigcap_{i=1}^k \bigcap_{j=1}^{N_i} f^{-\sum_{l=0}^{i-1} N_l n_l - (j-1) n_i} B_{n_i} (g; x^i_j, \frac{\epsilon}{2^i}) \neq \emptyset.
\]
We define $F_k$ by
\[
F_k = \{\overline{B(\ul x_1, \ldots, \ul x_k)}: (\ul x_1, \ldots \ul x_k) \in S_1^{N_1} \times \ldots \times S_k^{N_k}\}.
\]
Note that $F_k$ is compact and $F_{k+1} \subseteq F_k$. Define $F = \bigcap_{k=1}^{\infty} F_k$. 

\bl \label{lemma6}
For any $p \in F$, the sequence $\frac{1}{t_k} \sum_{i = 0}^{t_k-1} \varphi (f^i (p))$ diverges, where $t_k = \sum_{i=0}^{k} N_i n_i$.
\el
\bp
Choose $p \in F$ and let $p_k := f^{t_{k-1}}p$. Then there exists $(x_1^k, \ldots, x_{N_k}^k) \in S_k^{N_k}$ such that
\[
p_k \in \bigcap_{j=1}^{N_k} f^{-(j-1)n_k} \overline{B_{n_k}}(g; x^k_j, \epsilon/2_k).
\]
For $c >0$, let $\mbox{Var}(\varphi, c) := \sup \{ |\varphi(x) - \varphi(y)| : d(x,y)< c\}$. We have
\[
S_{n_k N_k} \varphi (p_k) \leq \sum_{j=1}^{N_k} S_{n_k} \varphi (x_j^k) + n_k N_k \mbox{Var} (\varphi, \epsilon/2^k) + N_kg(n_k, \epsilon/2^k) \| \varphi\|
\]
and hence
\[
\frac{1}{{n_k N_k}} S_{n_k N_k} \varphi (p_k) \leq \int \varphi d \mu_{\rho(k)} + \delta_k  + \mbox{Var} (\varphi, \epsilon/2^k) + \frac{1}{n_k}g(n_k, \epsilon/2^k). 
\]
It follows that
\[
\left | \frac{1}{{n_k N_k}} S_{n_k N_k} \varphi (p_k) - \int \varphi d \mu_{\rho(k)} \right| \ra 0. 
\]
We can use the fact that $\frac{n_k N_k}{t_k} \ra 1$ to prove that
\[
\left | \frac{1}{{n_k N_k}} S_{n_k N_k} \varphi (p_k) - \frac{1}{t_k} S_{t_k} \varphi (p) \right | \ra 0,
\]
and the result follows.
\ep
\subsection{Construction of a special sequence of measures $\mu_k$} \label{ss}
We must first undertake an intermediate construction. For each $\ul x =(\ul x_1, \ldots, \ul x_k) \in S_1^{N_1} \times \ldots \times S_k^{N_k}$, we choose one point $z = z(\ul x)$ such that
\[
z \in B(\ul x_1, \ldots \ul x_k).
\]
Let $\mc T_k$ be the set of all points constructed in this way. We show that points constructed in this way are distinct and thus $\#\mc T_k = \#S_1^{N_1} \ldots \#S_k^{N_k}$. 
\bl \label{lb2}
Let $\ul x$ and $\ul y$ be distinct members of $S_1^{N_1} \times \ldots \times S_k^{N_k}$. Then $z_1 := z ( \ul x )$ and $z_2 := z( \ul y )$ are distinct points. Thus $\#\mc T_k = \#S_1^{N_1} \ldots \#S_k^{N_k}$. 
\el
\begin{proof} 
Since $\ul x \neq \ul y$, there exists $i,j$ so $x^i_j \neq y^i_j$. We have $\Lambda_1, \Lambda_2 \in I(g; n_i, \epsilon/2^i)$ such that
\[
d_{\Lambda_1} (x^i_j, f^{a} z_1 ) < \frac{\epsilon}{2^i}  \mbox{ and }  d_{\Lambda_2} (y^i_j, f^{a} z_2 ) < \frac{\epsilon}{2^i},
\] 
where $a= \sum_{l=0}^{i-1} N_l n_l + (j-1) n_i$. Let $\Lambda = \Lambda_1 \cap \Lambda_2$. Since $\Lambda \in I(2g; n_i, \epsilon/2^i)$, we have $d_\Lambda (x^i_j , y^i_j) >4 \epsilon$. Using these inequalities, we have $d_\Lambda (f^{a} z_1 , f^{a} z_2) >3 \epsilon$.
\end{proof}

We now define the measures on $F$ which yield the required estimates for the Entropy Distribution Principle. We define, for each $k$, an atomic measure centred on $\Tk$. Precisely, let
\[
\nu_k := \sum_{z \in \Tk}\delta_z  
\]
We normalise $\nu_k$ to obtain a sequence of probability measures $\mu_k$, ie. we let $\mu_k := \frac{1}{\# \mc T_k} \nu_k$.
\bl \label{lemma5}
Suppose $\mu$ is a limit measure of the sequence of probability measures $\mu_k$. 
Then $\mu (F) = 1$.
\el
\bp
For any fixed $l$ and all $p \geq 0$, $\mu_{l+p} (F_{l}) = 1$ since $\mu_{l+p} (F_{l+p}) = 1$ and $F_{l+p} \subseteq F_{l}$. 
Suppose $\mu = \lim_{k \ra \infty} \mu_{l_k}$ for some $l_k \ra \infty$, then $\mu(F_l) \geq \limsup_{k \ra \infty} \mu_{l_k} (F_l) = 1$. It follows that $\mu (F) = \lim_{l \ra \infty} \mu (F_{l}) = 1$.
\ep
In fact, the measures $\mu_k$ converge. However, by using our version of the Entropy Distribution Principle, we do not need to use this fact and so we omit the proof (which goes like lemma 5.4 of \cite{TV}).  

Let $\mc B := B_n (q, \epsilon)$ be an arbitrary ball which intersects $F$. Let $k$ be the unique number which satisfies $t_k \leq n < t_{k+1}$. Let $j \in \{0, \ldots, N_{k+1} -1 \}$ be the unique number so
\[ 
t_k + n_{k+1}j \leq n < t_k + n_{k+1}(j+1).
\]
We assume that $j \geq 1$ and leave the details of the simpler case $j=0$ to the reader. The following lemma reflects the restriction on the number of points that can be in $\mc B \cap \mc T_{k+p}$.
\bl\label{lemma7}
For $p \geq 1$, $\mu_{k+p} ( \mc B) \leq (\# \Tk )^{-1} (\#\mc S_{k+1})^{-j}$
\el

\begin{proof} 
First we show that $\mu_{k+1} ( \mc B) \leq (\# \Tk )^{-1} (\#\mc S_{k+1})^{-j}$. We require an upper bound for the number of points in $\Tkk \cap \mc B$.
If  $\mu_{k+1} (\mc B) > 0$, then $\Tkk \cap \mc B \neq \emptyset$. Let $z = z(\ul x, \ul x_{k+1}) \in \Tkk \cap \mc B$ where $\ul x \in S_1^{N_1} \times \ldots \times S_k^{N_k}$ and $\ul x_{k+1} \in S_{k+1}^{N_{k+1}}$. Let
\[
\mc A_{\ul x; x_1, \ldots, x_j} = \{ z(\ul x, y_1, \ldots, y_{N_{k+1}}) \in \Tkk : x_1 = y_1, \ldots, x_j = y_j \}.
\]
We suppose that $z^\prime = z(\ul y, \ul y_{k+1})\in \Tkk \cap \mc B$ and show that $z^\prime \in \mc A_{x; x_1, \ldots, x_j}$. We have $d_n(z, z^\prime)<2 \epsilon$ and we show that this implies $x_l = y_l$ for $l \in \{1, 2, \ldots, j\}$ (the proof that $\ul x = \ul y$ is similar). 
Suppose that $y_l \neq x_l$ and let $a_l = t_k + (l-1)(n_{k+1})$. There exists $\Lambda_1, \Lambda_2 \in I(g; n_{k+1}, \epsilon /2^{k+1})$ such that 
\[
d_{\Lambda_1} (f^{a_l} z, x_l ) < \frac{\epsilon}{2^{k+1}}  \mbox{ and }  d_{\Lambda_2} (f^{a_l} z^\prime, y_l ) < \frac{\epsilon}{2^{k+1}}.
\] 
Let $\Lambda = \Lambda_1 \cap \Lambda_2$. Since $\Lambda \in I(2g; n_{k+1}, \epsilon/2^{k+1})$, we have $d_\Lambda (x_l , y_l) >4 \epsilon$.
We have
\beq
d_n (z, z^\prime) &\geq& d_\Lambda (f^{a_l}z, f^{a_l}z^\prime) \\
&\geq& d_\Lambda (x_l, y_l) - d_\Lambda (f^{a_l}z, x_l) - d_\Lambda (f^{a_l}z^\prime, y_l) \geq 3\epsilon,
\eeq
which is a contradiction.
Thus, we have
\beq
\nu_{k+1} (\mc B) \leq \# \mc A_{x; x_1, \ldots, x_j} = (\#S_{k+1})^{N_{k+1} - j},
\eeq
\[
\mu_{k+1} (\mc B) \leq (\# \mc T_{k+1})^{-1} (\#S_{k+1})^{N_{k+1} - j} = (\# \Tk )^{-1} (\#\mc S_{k+1})^{-j}
\]
Now consider $\mu_{k+p} ( \mc B)$. Arguing similarly to above, we have
\beq
\nu_{k+p} ( \mc B)& \leq &\# \mc A_{x; x_1, \ldots, x_j} (\# \mc S_{k+2})^{N_{k+1}} \ldots (\# \mc S_{k+p})^{N_{k+p}} 
\eeq
The desired result follows from this inequality by dividing by $\# \mc T_{k+p}$.
\end{proof}

By lemma \ref{lemma1}, we have 
\beq
\# \mc T_k (\# \mc S_{k+1})^j &\geq& \exp \{ (\htop (f)-4 \gamma) (N_1 n_1 + N_2 n_2 + \ldots + N_k n_k +j n_{k+1}) \}\\
&\geq& \exp \{ (\htop (f)-4 \gamma) n \}.
\eeq
Combining this with the previous lemma gives us
\[
\limsup_{l \ra \infty} \mu_{l} (B_n (q, \epsilon)) \leq \exp \{-n(\htop (f) - 4 \gamma)\}.
\] 
Applying the Entropy Distribution Principle, 
we have
\[
\htop(F, \epsilon) \geq \htop (f) - 4 \gamma.
\] 
Since $\gamma$ and $\epsilon$ were arbitrary and $F \subset \widehat X(\varphi, f)$, we have $\htop (\widehat X(\varphi, f)) = \htop (f)$.
\ep
\section{The $\beta$-transformation} \label{beta.4}
In this section, let $X = [0,1)$. For any fixed $\beta >1$, we consider the $\beta$-transformation $f_\beta: X \mapsto X$ given by
\[
f_\beta(x) = \beta x \mbox{ }(\mbox{mod}1).
\]
As reference for the basic properties of the $\beta$-transformation, we recommend the introduction of the thesis of Maia \cite{Maia}. For $\beta \notin \IN$, let $b = [\beta]$ and for $\beta \in \IN$, let $b = \beta-1$. We consider the partition into $b+1$ intervals
\[
J_0 = \left[0, \frac{1}{\beta} \right), J_1 = \left[\frac{1}{\beta}, \frac{2}{\beta} \right), \ldots, J_b = \left[\frac{b}{\beta}, 1 \right).
\]
For $x \in [0,1)$, let $w(x) =(w_j(x))_{j=1}^\infty$ 
be the sequence given by $w_j (x) = i$ when $f^{j-1} x \in J_i$. We call $w(x)$ the greedy $\beta$-expansion of $x$ and we have
\[
x = \sum_{j=1}^{\infty} w_j (x) \beta^{-j}.
\]
The $\beta$-shift $(\Sigma_\beta, \sigma_\beta)$ is the subshift defined by the closure of all such sequences in $\prod_{i=1}^\infty \{0, \ldots, b \}$. 
Let $w(\beta) = (w_j(\beta))_{j=1}^\infty$ denote the sequence which is the lexicographic supremum of all $\beta$-expansions. 
The sequence $w(\beta)$ satisfies
\[
\sum_{j=1}^{\infty} w_j (\beta) \beta^{-j} = 1,
\]
so we call $w(\beta)$ the $\beta$-expansion of $1$. Parry showed that the set of sequences which belong to $\Sigma_\beta$ can be characterised as 
\[
w \in \Sigma_\beta \iff \sigma^k (w) \leq w(\beta) \mbox{ for all } k \geq 1,
\]
where $\leq$ is taken in the lexicographic ordering \cite{Pa}. Parry also showed that any sequence $w$ which satisfies $\sigma^k (w) \leq w$ is the $\beta$-expansion of $1$ for some $\beta>1$. 
The $\beta$-shift contains every sequence which arises as a greedy $\beta$-expansion and an additional point for every $x$ whose $\beta$-expansion is finite (i.e. when there exists $j$ so $w_i(x) = 0$ for all $i \geq j$). Thus the map $\pi: \Sigma_\beta \mapsto [0,1]$ defined by
\[
\pi( \ul w) = \sum_{j=1}^{\infty} w_j \beta^{-j}
\] 
is one to one except at the countably many points for which the $\beta$-expansion is finite. 

$\Sigma_\beta$ is typically not a shift of finite type (nor even a shift with specification) and the set of all $\beta$-shifts gives a natural and interesting class of subshifts. In the next section, we decribe in detail the known results on specification properties for the $\beta$-shift. The key fact for our analysis is that every $\beta$-shift has the almost specification property \cite{PfS2}. 
We have (p.179 of \cite{Wa}) that $\htop(\sigma_\beta) = \log \beta$.

\begin{theorem} \label {theorem.a.0.2}
For $\beta >1$, let $f_\beta: X \mapsto X$ be the $\beta$-transformation, $f_\beta (x) = \beta x (\mod 1)$. Let $\varphi \in C([0,1])$ and assume that the irregular set for $\varphi$ is non-empty (ie. $\widehat {X} (\varphi, f_\beta) \neq \emptyset$), then $h_{top} (\widehat {X} (\varphi, f_\beta)) = \log \beta$.
\end{theorem}
\bp
Let $\Sigma_\beta^\prime$ denote the set of sequences which arise as $\beta$-expansions. Recall that $\Sigma_\beta \setminus \Sigma_\beta^\prime$ is a countable set and the restriction of $\pi$ to $\Sigma_\beta^\prime$ is a homeomorphism satisfying $\pi \circ \sigma_\beta = f_\beta \circ \pi$. 
Thus, if $Z \in \Sigma_\beta$ and $Z^\prime := Z \cap \Sigma_\beta^\prime$, we have
\[
\htop(Z, \sigma_\beta) =\htop(Z^\prime, \sigma_\beta) = \htop (\pi(Z^\prime), f_\beta).
\]
Suppose $\varphi:[0,1] \mapsto \IR$ satisfies $\widehat X(\varphi, f_\beta) \neq \emptyset$. Let $x \in \widehat X(\varphi, f_\beta)$ and let $w(x)$ be its $\beta$-expansion. We let $\overline \varphi \in C(\Sigma_\beta)$ be the unique continuous function which satisfies $\overline \varphi = \varphi \circ \pi$ on $\Sigma_\beta^\prime$ (this exists because we assumed $\varphi$ to be continuous on $[0,1]$). It is clear that $w(x) \in \widehat {\Sigma_\beta}(\overline \varphi, \sigma_\beta)$. Since the dynamical system $(\Sigma_\beta, \sigma_\beta)$ satisfies the almost specification property, it follows from theorem \ref{theorem.a.0.2} that $\htop (\widehat {\Sigma_\beta}(\overline \varphi, \sigma_\beta)) = \log \beta$. Since $\pi (\widehat {\Sigma_\beta}(\overline \varphi, \sigma_\beta) \cap \Sigma_\beta^\prime) = \widehat X(\varphi, f_\beta)$, it follows that $\htop (\widehat X(\varphi), f_\beta) = \log \beta$.
\ep
\subsection{$\beta$-transformations and specification properties} \label{beta.4.1}
There is a simple presentation of $\Sigma_\beta$ by a labelled graph $\mc G_\beta$ due to Blanchard and Hansel \cite{BH}.  We describe the construction of $\mc G_\beta$ when the $\beta$-expansion of $1$ is not eventually periodic. We refer the reader to \cite{PfS2} for the slightly different construction required when the $\beta$-expansion of $1$ is eventually periodic (in this case, $\Sigma_\beta$ is a sofic shift \cite{BM} and therefore has specification).

Let $v_1, v_2, \ldots$ be a countable set of vertices. We draw a directed edge from $v_i$ to $v_{i+1}$ and label it with the value $w_i(\beta)$ for all $i \geq1$. 
If $w_i(\beta) \geq 1$, we draw a directed edge from $v_i$ to $v_{1}$ labelled with the value $0$. If $b=1$, the construction is complete. If $b>1$, then for all $j \in \{2, \ldots, b\}$ and all $w_i(\beta) \geq j$, we draw a directed edge from $v_i$ to $v_{1}$ labelled with the value $j-1$. Note that if $w_i(\beta)=0$, the only edge which starts at $v_i$ is the edge from $v_i$ to $v_{i+1}$ labelled by $0$, and if $w_i(\beta)\neq0$ there is always an edge from $v_i$ to $v_1$.
We have $w \in \Sigma_\beta$ iff $w$ labels an infinite path of directed edges of $\mc G_\beta$ which starts at the vertex $v_1$. The following figure depicts part of the graph $\mc G_\beta$ for a value of $\beta$ satisfying $(w_j(\beta))_{j=1}^6 = (2, 0,1,0,0,1)$.

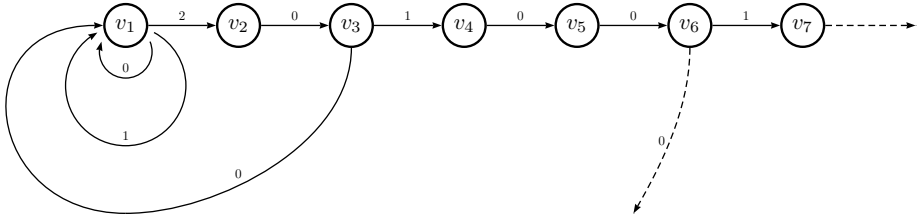
\begin{figure}[h]
       \begin{center}
                          \LargeState\SmallPicture\VCDraw{
                      \begin{VCPicture}{(-2,-7)(20,2)}
                                
                               \State[v_1]{(0,0)}{v1}
                                \State[v_2]{(3,0)}{v2}
                                \State[v_3]{(6,0)}{v3}
                               \State[v_4]{(9,0)}{v4}
                           \State[v_5]{(12,0)}{v5}
                               \State[v_6]{(15,0)}{v6}
                           \State[v_7]{(18,0)}{v7}
                      
                   \EdgeL{v1}{v2}{}\taput{2}
                               \EdgeL{v2}{v3}{}\taput{0}
                                \EdgeL{v3}{v4}{}\taput{1}
                               \EdgeL{v4}{v5}{}\taput{0}
                               \EdgeL{v5}{v6}{}\taput{0}
                                \EdgeL{v6}{v7}{}\taput{1}
                               
                                \cnode*(0,-5){0pt}{p0}
                               \cnode*(13.5,-5){0}{p1}
                                \nccurve[angleA=-90,angleB=60,linestyle=dashed]{->}{v6}
{p1}\taput[npos=.7]{0}
                               \nccircle[angleA=180,nodesep=5pt]{<-}{v1}{.7cm}\taput[npos=.7]{0}
                               \nccircle[angleA=180,nodesep=5pt]{<-}{v1}{1.6cm}\taput[npos=.7]{1}
                               \nccurve[angleA=-90,angleB=0]{-}{v3}{p0}\taput[npos=.7]{0}
                           \nccurve[angleA=180,angleB=180,ncurv=1.5]{->}{p0}{v1}
                             \cnode*(21,0){0pt}{p3}
                           \ncline[linestyle=dashed]{v7}{p3}
                                
                     \end{VCPicture}}
                
        \end{center}
        
 \end{figure}


An arbitrary subshift $\Sigma$ on $b+1$ symbols is a closed shift-invariant subset of $\prod_{i=1}^{\infty} \{0,\ldots, b\}$.
We define $v$ to be admissible word of length $n \geq 1$ for $\Sigma$ if there exists $x\in \Sigma$ such that $v = (x_1, \ldots, x_{n})$. 
The specification property of definition \ref{3a} can easily be seen to be equivalent to the following property in the case of an arbitrary subshift. 
\begin{definition}
A subshift $\Sigma$ has the specification property if there exists $M>0$ such that for any two admissible words $w_1$ and $w_2$, there exists a word $w$ of length less than $M$ such that $w_1 w w_2$ is an admissible word. 
\end{definition}
We now return to the $\beta$-shift $\Sigma_\beta$. Define 
\[
z_n (\beta) = \min \{i \geq 0: w_{n+i} (\beta) \neq 0 \}.
\]
Equivalently, $z_n (\beta) +1$ is the minimum number of edges required to travel from $v_n$ to $v_1$. The $\beta$-shift fails to have the specification property iff `blocks of consecutive zeroes in the $\beta$ expansion of $1$ have unbounded length', ie. if $z_n (\beta)$ is unbounded \cite{BM}. Consider concatenations of the word $c_n:= (w_1(\beta), \ldots, w_n(\beta))$ with some other admissible word $v$. We can see from the graph $\mc G_\beta$ that the length of the shortest word $w$ such that $c_n w v$ is an admissible word is $z_n(\beta)$ (the word $w$ is a block of zeroes of length $z_n(\beta)$). Now for $x \in \Sigma_\beta$, we define $z_n (x) \geq 0$ to be the length of the shortest word $w$ required so that for any admissible word $v$, $(x_1, \ldots, x_n) w v$ is an admissible word. Note that for all $x \in \Sigma_\beta$, $z_n( x) \leq \max\{ z_i(\beta) : 1\leq i \leq n\}$. Thus $\Sigma_\beta$ has specification iff $z_n(\beta)$ is bounded. Buzzi shows that the set of $\beta$ for which this situation occurs has Lebesgue measure $0$. 

Pfister and Sullivan \cite{PfS2} used the graph $\mc G_\beta$ to observe that every $\beta$-shift has the almost specification property. Their strategy is to `jump ship' on the last non-zero entry of an admissible word. More precisely, every $\beta$-shift has the following property. Given any admissible word $w$, there is a word $w^\prime$ which differs from $w$ only by one symbol, such that $w^\prime v$ is admissible for every admissible word $v$. The modified word  $w^\prime$ is given by replacing the last non-zero entry of the word $w$ by a $0$. This property is best seen from inspection of the graph $\mc G_\beta$ and is the content of proposition 5.1 of \cite{PfS2}. It can easily be seen that this property implies the almost specification property. 



\subsection{Hausdorff dimension of the irregular set for the $\beta$-shift} We equip $\Sigma_\beta$ with the metric $d_\beta(x, y) = \frac{1}{\beta^n}$ where $n$ is the smallest integer such that $x_{n+1} \neq y_{n+1}$.  
Let $C_n (x) = \{ y \in \Sigma_\beta : x_i = y_i \mbox{ for } i = 0, \ldots, n-1\}$.  For $Z \subset \Sigma_\beta$, we define the diameter of the set to be $\mbox{Diam}(Z) = |Z| = \sup \{d_\beta (x, y) : x,y \in Z\}$. We have
\[
\mbox{Diam}(C_n (w(\beta))) = \frac{1}{\beta^{n + z_{n+1}(\beta)}},
\]
\[
\frac{1}{\beta^{n}} \geq \mbox{Diam}(C_n (x)) \geq \frac{1}{\beta^{n + z_{n+1}(x)}}. 
\]
Because there is no uniform lower bound on the diameter of a cylinder when $z_n(\beta)$ is unbounded, it is plausible that the relation between Hausdorff dimension and topological entropy could not be straightforward. Indeed, because of this potential difficulty, in \cite{PfS2, PfS}, it seems that Pfister and Sullivan establish a relationship between topological entropy and Hausdorff dimension only for $\beta$-shifts which satisfy an additional hypothesis on $w(\beta)$ (which holds for Lebesgue almost every $\beta$).  In fact,  the following lemma tells us that the relationship  between Hausdorff dimension and topological entropy is straightforward for every $\beta$-shift. While this lemma  is surely known to some experts as a folklore theorem and is a corollary of more general results of Climenhaga \cite[Theorem 2.4]{Cl},  we include an elementary direct proof to clarify the situation.  For Hausdorff dimension, we fix the notation

\[
H(Z, \alpha, \delta) = \inf \{ \sum |B(x_i, \delta_i)|^\alpha : Z \subseteq \bigcup_i B(x_i, \delta_i), \delta_i \leq \delta \}.\,
\]
$H(Z, \alpha) = \lim_{\delta \ra 0} H(Z, \alpha, \delta)$ and $\dimh (Z) = \sup \{ \alpha :H(Z, \alpha) = \infty\}$. We sometimes write $\dimh (Z, d)$ in place of $\dimh(Z)$ when we wish to emphasise the dependence on the metric $d$. We note that the map $x \mapsto w(x)$ is bi-Lipshitz with respect to the metric $d_\beta$ and thus for $Z \subset [0,1)$, $\dimh (Z) = \dimh (\pi^{-1} (Z), d_\beta)$.

\bl \label{dimhbeta}
For arbitrary $Z \subset \Sigma_\beta$, we have $ \log \beta \dimh (Z) = \htop(Z)$.
\el
\bp
That $\log \beta \dimh (Z) \leq \htop(Z)$ is a standard argument which follows from the fact that $\mbox{Diam}(C_n (x)) \leq \frac{1}{\beta^{n}} $.  For the other inequality, let $\gamma_N = \beta^{-N}$, and consider a cover of $Z$ by metric balls $B(x_i, \delta_i)$ with $\delta_i < \gamma_N$. We modify the radii of the balls to create a `better' cover.  For each $\delta_i$, let $n_i$ be the unique natural number so that $\beta^{-(n_i+1)} < \delta_i \leq \beta^{n_i}$. Then $n_i \geq N$ for all $i$ and $|B(x_i, \delta_i)| \leq \beta^{-(n_i+1)}$. If $|B(x_i, \delta_i)| = \beta^{-(n_i+1)}$, let $\rho_i = \delta_i$ and $m_i=1$. If not, $|B(x_i, \delta_i)| = \beta^{-(n_i+m_i)}$ for some $m_i \geq2$. In this case, let $\rho_i = \beta^{-(n_i+m_i-1)}$. We have $\rho_i \leq \delta_i$ and $B(x_i, \delta_i) = B(x_i, \rho_i) = C_{n_i+m_i}(x)$. The collection $\Gamma  = \{C_{n_i+m_i}(x_i)\}$ covers $Z$, and thus
\begin{align*}
\sum  |B(x_i, \delta_i)|^\alpha = \sum  |B(x_i, \rho_i)|^\alpha = \sum \beta^{-(n_i+m_i)\alpha} &= \sum \exp (-\alpha (n_i + m_i) \log \beta) \\ & = Q(Z, \alpha \log \beta, \Gamma).
\end{align*}
Taking infimums, we have $H(Z, \alpha, \gamma_N) \geq M(Z, \alpha \log \beta, N) $. Taking the limit as $N \ra \infty$, it follows that if $m(Z, \alpha \log \beta) = \infty$, then $H(Z, \alpha) = \infty$ and the inequality follows.
\ep
We obtain the following theorem as an immediate corollary of lemma \ref{dimhbeta} and theorems \ref {theorem.a.0.1} and \ref{theorem.a.0.2}.

\bt \label{theorem.ae.beta}
For every $\beta>1$, 

(1) If $\varphi \in C(\Sigma_\beta)$ and $\widehat \Sigma_\beta (\varphi, \sigma_\beta) \neq \emptyset $, then $ \dimh (\widehat \Sigma_\beta (\varphi, \sigma_\beta)) =1$,

(2) If $\varphi \in C([0,1])$ and $\widehat X (\varphi, f_\beta) \neq \emptyset$, then $ \dimh ( \widehat X (\varphi, f_\beta)) =1$.
\et

\begin{remark} 
We may also consider the Billingsley dimension $\mbox{Dim}_{\nu}$ of the irregular set (with respect to a reference measure $\nu$) \cite{PfS2, Caj}. We remark that when $\nu$ is equivalent to Lebesgue measure, then $\mbox{Dim}_{\nu} = \dimh$. Every $\beta$-transformation has an invariant measure $\nu_\beta$ which is equivalent to Lebesgue (and is the measure of maximal entropy). Therefore, it is a corollary of theorem \ref{theorem.ae.beta} that if $\varphi \in C(\overline X)$ and $\widehat X(\varphi, f_\beta) \neq \emptyset$, then $\mbox{Dim}_{\nu_\beta} (\widehat X(\varphi, f_\beta)) =1$.
\end{remark}

 

\subsection{An alternative approach for $\beta$-shifts} \label{alternative}
We describe an alternative proof of Theorem \ref{theorem.ae.beta} based on approximating $\beta$-shifts from inside by $n$-step shifts of finite type. We discuss the relative merits of the two approaches in \S \ref{sfc}. The key quoted result in this method is a version of theorem \ref{theorem.a.0.2} in the special case of $n$-step Markov shifts. We note that the `almost specification' method of proof applies in far greater generality and a self-contained version of the proof described below would be comparable in length.  

\bp
Recall that any sequence $(a_n)$ on a finite number of symbols which satisfies $\sigma^k (a_n) \leq (a_n)$ for all $k \geq 0$ arises as $w (\beta)$ for some $\beta >1$.  Fix $\beta >1$ and write $w_i := w_i (\beta)$. Let $\beta(n)$ be the simple $\beta$-number corresponding to the sequence $(w_1, w_2, \ldots, w_n, 0, 0, 0, \ldots)$.
An elementary argument \cite{Pa} shows that $\beta(n) \ra \beta$. It is clear that $\Sigma_{\beta(n)}$ can be considered to be a subsystem of $\Sigma_\beta$ (the subshift $\Sigma_{\beta(n)}$ corresponds to the set of labels of edges of infinite paths that only visit the first $n$ vertices of $\mc G_\beta$).

Now suppose $\varphi \in C(\Sigma_\beta)$ is a function for which the irregular set is non-empty. Then there exists $x, y \in \Sigma_\beta$ such that 
\[
\limn \frac{1}{n} S_n \varphi(x) \neq \limn \frac{1}{n} S_n \varphi(y).
\]
Let $\delta>0$ and $N_1 \in \IN$ be such that for $n \geq N_1$,
\[
\left |\frac{1}{n} S_n \varphi(x) - \frac{1}{n} S_n \varphi(y) \right | > 4\delta.
\]
Pick $N_2$ sufficiently large that
\[
\sup \{|\varphi(w)-\varphi(v)|: w,v \in \Sigma_\beta,  w_i =v_i \mbox{ for } i=1, \ldots, N_2\} < \delta. 
\]
For any $n \geq N = \max\{N_1, N_2\}$, let us choose $x^\prime \in$ $C_N (x) \cap \Sigma_{\beta(n)}$ and $y^\prime \in$ $C_N(y) \cap \Sigma_{\beta(n)}$. We have for all $m \geq N$,
\[
\left |\frac{1}{m} S_m \varphi(x^\prime) - \frac{1}{m} S_m \varphi(y^\prime) \right | > 2\delta.
\]
Thus the restriction of $\varphi$ to $\Sigma_{\beta(n)}$ does not have trivial spectrum of Birkhoff averages and by lemma \ref{equiv}, 
our main theorem gives us
\begin{equation} \label{5.1}
\htop(\widehat \Sigma_{\beta(n)} (\varphi, \sigma_{\beta(n)})) = \htop(\sigma_{\beta(n)}) = \log \beta(n).
\end{equation}
We remark that $\Sigma_{\beta(n)}$ is an $n$-step Markov shift (and thus has specification), 
so formula (\ref{5.1}) also follows for H\"older continuous $\varphi$ from theorem 9.3.2 of \cite{Bad}. Note that $\widehat \Sigma_{\beta(n)} (\varphi, \sigma_{\beta(n)}) \subset \widehat \Sigma_\beta (\varphi, \sigma_\beta)$.

By lemma \ref{dimhbeta}, any subset $Z \subset \Sigma_{\beta(n)}$ satisfies $\dimh (Z, d_\beta) = \htop(Z) / \log \beta$. 
In particular, $\dimh( \Sigma_{\beta(n)}, d_\beta) = \log (\beta(n)) / \log \beta$.
Thus
\[
\dimh(\widehat \Sigma_\beta (\varphi, \sigma_\beta), d_\beta) \geq \sup \left \{ \dimh (\widehat \Sigma_{\beta(n)} (\varphi, \sigma_{\beta(n)}), d_\beta)\right \} = \sup \left \{\frac{\log \beta(n)}{\log \beta}\right \} = 1.
\]
\ep

\subsection{Some final comments} \label{sfc}
We give some closing remarks, comparing the `almost specification' method of theorem \ref{theorem.a.0.1} with the `approximation from inside' method of \S \ref{alternative}. 
\begin{remark}
An almost sofic shift \cite{LM} is defined to be a shift space $\Sigma$ for which one can find a sequence of subshifts of finite type $\Sigma_n$ such that $\Sigma_n \subset \Sigma$ and $\limn \htop (\Sigma_n, \sigma) = \htop (\Sigma, \sigma)$. By our previous reasoning, every $\beta$-shift is almost sofic. We remark that the proof of theorem \ref{alternative} shows that if $(\Sigma, \sigma)$ is an almost sofic shift, $\varphi \in C(\Sigma)$ and $\widehat \Sigma (\varphi, \sigma) \neq \emptyset $, then $ \htop (\widehat \Sigma (\varphi, \sigma)) = \htop (\sigma)$.
\end{remark}
\begin{remark} 
We point out that theorem \ref{theorem.a.0.1} applies to systems where the approach of \S \ref{alternative} cannot apply, even in the setting of shift spaces. It is not too difficult to construct a shift space $(\Sigma, \sigma)$ such that

(1) $(\Sigma, \sigma)$ has positive topological entropy but no periodic points;

(2) $(\Sigma, \sigma)$ satisfies almost specification but not specification. 

We give the idea of the construction. We define $\Sigma := \bigcap_{n \geq1} \Sigma_n$, where $\Sigma_n \subset \prod_{i=1}^\infty \{0, 1\}$ are suitably chosen topologically mixing shifts of finite type satisfying $\Sigma_{n+1} \subset \Sigma_n$ for all $n \geq 1$. For $i \geq 1$, let $\epsilon_i = \frac{1}{2^{i+1}}\log 2$ and $N_i$ be an increasing sequence of integers to be chosen later.

Recall that a finite collection of forbidden words defines a shift of finite type. For a word $u$ and an integer $N$, we use the notation $u^N$ to denote the concatenation of $N$ copies of $u$. Let $\mc F_1$ denote the set of words
\[
\mc F_1 = \{(1)^{N_1}, (0)^{N_1}\}
\]
and define $\Sigma_1$ as the shift defined by the forbidden words $\mc F_1$. It follows easily from theorem 7.13 of \cite{Wa} that if we choose $N_1$ to be large enough, then $\htop (\Sigma_1, \sigma) > \log 2 - \epsilon_1$. 
It is clear that $\Sigma_1$ contains no fixed points. 

We define $\Sigma_k$ inductively as the shift defined by the collection of forbidden words $\bigcup_{i=1}^k \mc F_i$, where
\[
\mc F_k = \{ v^{N_k} : \mbox{ v is an admissible word of length $k$ in $\Sigma_{k-1}$ } \},
\]
with $N_k$ chosen large enough so that $\htop (\Sigma_k, \sigma) > \htop (\Sigma_{k-1}, \sigma) - \epsilon_i$. It is clear that $\Sigma_k$ contains no periodic points of period less than or equal to $k$. 

It is easily verified that $\htop( \Sigma, \sigma) = \limk \htop (\Sigma_k, \sigma) >0$. Since $\Sigma$ contain no periodic points, it contains no shifts of finite type. We claim that $(\Sigma, \sigma)$ satisfies almost specification but not specification and briefly explain why. The key fact is that if a word $w$ belongs to $\mc F_k$, we can change it to a word not in $\mc F_k$ by changing any one of its entries (from a $0$ to a $1$ or vice-versa). Furthermore, one can modify $w$ by changing a single entry in at least $|w|(1 - 2/N_1)$ positions so that the new word is not in $\bigcup_{i=1}^k \mc F_i$. This abundance of choice when we wish to adjust an inadmissible word to make it admissible gives rise to the almost specification property. 

It is not hard to find $\Sigma$-admissible words $u$ and $w_k$ such that $u v w_k$ is $\Sigma$-inadmissible for any word $v$ of length less than or equal to $k$. Hence, $(\Sigma, \sigma)$ does not have specification. However, by changing $u$ and $w_k$ at a small number of entries (relative to the length of $u$ and $w_k$), we can easily find new words $u^\ast$ and $w_k^\ast$ such that $u^\ast w_k^\ast$ is $\Sigma$-admissible. This leads to the almost specification property. 
\end{remark}

\begin{remark} 

We could study the class of systems $(X, f)$ such that for every $\epsilon >0$, there exists $n \in \IN$ and a compact $f^n$-invariant subset $Y \subset X$ satisfying

(1) $(Y, f^n)$ is a topological factor of a shift of finite type;

(2) $\htop (Y, f^n) > n (\htop (X, f) -\epsilon)$. 

We call $(Y, f^n)$ a horseshoe for $(X, f)$. 
To study irregular sets for $(X, f)$, it should suffice to study the intersection of the irregular set with the horseshoe. As demonstrated by our previous remark, systems with almost specification do not necessarily contain any horseshoes, so we could not prove theorem \ref{theorem.a.0.1} using this approach. Also, theorems on the existence of horseshoes typically require smoothness of the system (see, for example, theorem S.5.9 of \cite{KH}), whereas our `almost specification' approach is a topological approach to a topological question. However, we do note that examples exist that do not have specification but where an `approximation from inside' approach could yield results. For example, a continuous interval map which is not mixing contains horseshoes but does not have specification (see corollary 15.2.10 of \cite{KH}). Thus, the `almost specification' approach and the `approximation from inside' approach both have their own merits.  
\end{remark}
\begin{remark}
In \cite{Pet}, Petersen studies some interesting families of shift spaces which first arose in an applied setting. He shows that some of these are almost sofic and some are not. It would be interesting to investigate the almost specification property in this setting. 
\end{remark} 
\section*{Acknowledgements}
This work constitutes part of my PhD, which was supported by the EPSRC. I would like to thank my supervisors Mark Pollicott and Peter Walters for many useful discussions and reading draft versions of this work, for which I am most grateful. I would also like to thank Thomas Jordan for comments on a draft version of this work and Charles-Edouard Pfister for a useful discussion.

\bibliographystyle{amsplain}
\bibliography{master}



\end{document}